\documentclass[12pt]{amsart}
\usepackage{amsmath}
\usepackage{amssymb}
\usepackage{amsthm}
\usepackage{a4wide}
\usepackage{xcolor}
\usepackage{tikz}
\newtheorem{thm}{Theorem}
\newtheorem{lem}[thm]{Lemma}

\begin{document}

\title{Fair partitions of the plane into incongruent pentagons}
\author{Dirk Frettl{\"o}h}
\address{Faculty of Technology, Bielefeld University, 33501 Bielefeld, Germany}
\author{Christian Richter}
\address{Institute for Mathematics, Friedrich Schiller University, 07737 Jena, Germany}
\date{\today}

\begin{abstract}
Motivated by a question of R.\ Nandakumar, we show that the Euclidean plane can be dissected into mutually incongruent convex pentagons of the same area and the same perimeter.
\end{abstract}


\maketitle


\section{Background and main result}

In his fruitful mathematical blog R.\ Nandakumar posed the following question  \cite[post of December 10, 2014]{nblog2}. \emph{Can the plane $\mathbb{R}^2$ be tiled by triangles of same 
area and perimeter such that no two triangles are congruent to each other?} In the present paper a tiling of $\mathbb{R}^2$ always means a family of mutually non-overlapping polygons whose union is $\mathbb{R}^2$. We call a tiling fair if all its members have the same area and perimeter (cf. \cite{dmo}). Congruence is meant with respect to Euclidean isometries including reflections.
Kupavskii, Pach and Tardos \cite{kpt2} have shown that the answer to the above question is negative. 
Frettl\"oh \cite{f}, the first named author of the present paper, had already noted 
that the answer is negative when the tilings are vertex-to-vertex, i.e., when the intersection of any two tiles is either a common side, a common vertex, or empty.

Nandakumar proposed to weaken the condition of equality of perimeters \cite{nblog2}, and Frettl\"oh (and in parts also Nandakumar) opened the question by also considering tilings with convex $n$-gons, $n=4,5,6$ \cite{f}. This has triggered the construction of several families of tilings by mutually incongruent convex polygons of equal area \cite{f,kpt1,fr1,fr2}. See also Nandakumar's post of August 2, 2021, concerning tilings of the plane as well as the link there to previous posts \cite{nblog2}.

Given the non-existence of fair tilings of $\mathbb{R}^2$ by incongruent triangles from \cite{kpt2}, the authors of the present note constructed a fair tiling by incongruent convex quadrangles \cite{fr2}, which seems to be the first example of a fair tiling of $\mathbb{R}^2$ by convex $n$-gons. Here we solve the case $n=5$.  

%
%

\begin{thm} \label{thm:pentagons}
There is a fair tiling of $\mathbb{R}^2$ by pairwise incongruent convex pentagons.
\end{thm}

Our examples for the cases $n=4,5$ are not vertex-to-vertex. The following questions remain open. \emph{Are there fair vertex-to-vertex tilings by incongruent convex quadrangles (or pentagons)? Are there fair vertex-to-vertex tilings by incongrent convex quadrangles arbitrarily close to the vertex-to-vertex tiling by unit squares? Are there fair tilings by incongruent convex hexagons (vertex-to-vertex or not)? Are there fair tilings by incongruent convex hexagons arbitrarily close to the regular honeycomb tiling?} It seems to us that all these questions have positive answers. 


\section{Proof of Theorem~\ref{thm:pentagons}}

Our proof is based on the method of \emph{tiling a tile} \cite{f}: We start with a periodic tiling of the plane by clusters of seven regular hexagons of side length $1$; see the left-hand part of Figure~\ref{fig:1}. Lemma~\ref{lem:1} produces a certain tiling of the plane into hexagons by subdividing the clusters; see the right-hand part of Figure~\ref{fig:1}. Then Lemma~\ref{lem:2} splits each hexagon into three pentagons; see Figure~\ref{fig:2}.

Given $\varepsilon >0$, we call two real numbers $\varepsilon$-close if their absolute difference does not exceed $\varepsilon$.
Two convex $n$-gons are called $\varepsilon$-close if there is a bijection between their vertices such that the Euclidean distance between corresponding vertices is at most $\varepsilon$. 
Two tilings by convex $n$-gons are called $\varepsilon$-close if there is a bijection between them such that the distances between corresponding vertices of corresponding $n$-gons do not exceed $\varepsilon$.  

A side figure of a convex polygon consists of a side of that polygon together with the two adjacent inner angles. In particular, a side figure is determined up to congruence by the length of the side and the two sizes of the adjacent angles.

\begin{lem}\label{lem:1}
For every $\varepsilon > 0$, there exists a tiling of the plane such that
\begin{itemize}
\item[(i)]
the tiling is $\varepsilon$-close to a periodic tiling by regular hexagons of side length $1$,
\item[(ii)]
all tiles are hexagons of area $\frac{3\sqrt{3}}{2}$ (which is the area of a regular hexagon of side length $1$),
\item[(iii)]
in every tile three side figures over non-adjacent sides are marked,
\item[(iv)]
all marked side figures within the tiling are mutually incongruent.
\end{itemize}
\end{lem}

\begin{proof}
We start by arranging the periodic tiling by regular hexagons of side length $1$ into clusters of seven hexagons; see the left-hand part of Figure~\ref{fig:1}.
\begin{figure}
\begin{center}
\begin{tikzpicture}[xscale=.15,yscale=.1299]

\foreach \x in {0,1,2,3} {
  \foreach \y in {0,1,2} {
    \draw
      (3*\x+9*\y,10*\x+2*\y) node {
        \begin{tikzpicture}[xscale=.15,yscale=.1299]
\begin{scope}[cm={0,-1,-1,0,(0,0)}]
        \draw[line width=.5mm]
                  (2,1)--(4,0)--(6,1)--(8,0)--(10,1)--(10,3)--(12,4)--(12,6)--(10,7)--(10,9)--(8,10)--(6,9)--(4,10)--(2,9)--(2,7)--(0,6)--(0,4)--(2,3)--cycle
          ;
        \draw[thin]
          (2,3)--(4,4) (2,7)--(4,6)--(4,4)--(6,3)--(6,1) (4,6)--(6,7) (6,3)--(8,4) (6,9)--(6,7)--(8,6)--(8,4)--(10,3) (8,6)--(10,7)
        ;
\end{scope}
        \end{tikzpicture}
      }
    ;
  }
}

\draw
  (58,17) node {
    \begin{tikzpicture}[xscale=.6,yscale=.5196]
\begin{scope}[cm={0,-1,-1,0,(0,0)}]
    \fill[lightgray]
      (3.5,.25)--(4,0)--(6,1)--(6,1.5)--cycle
      (6,1.5)--(6,1)--(8,0)--(8.5,.25)--cycle
      (3.5,9.75)--(4,10)--(6,9)--(6,8.5)--cycle
      (6,8.5)--(6,9)--(8,10)--(8.5,9.75)--cycle
      (2,1)--(2,3)--(2.5,3.25)--(2.5,.75)--cycle
      (2,9)--(2,7)--(2.5,6.75)--(2.5,9.25)--cycle
      (10,1)--(10,3)--(9.5,3.25)--(9.5,.75)--cycle
      (10,9)--(10,7)--(9.5,6.75)--(9.5,9.25)--cycle
      (0,4.5)--(0,4)--(2,3)--(2.5,3.25)--cycle
      (0,5.5)--(0,6)--(2,7)--(2.5,6.75)--cycle
      (12,4.5)--(12,4)--(10,3)--(9.5,3.25)--cycle
      (12,5.5)--(12,6)--(10,7)--(9.5,6.75)--cycle
      (4,4)--(4,6)--(3.5,6.25)--(3.5,3.75)--cycle
      (8,4)--(8,6)--(8.5,6.25)--(8.5,3.75)--cycle
      (4,4)--(6,3)--(6,2.5)--(3.5,3.75)--cycle
      (4,6)--(6,7)--(6,7.5)--(3.5,6.25)--cycle
      (8,4)--(6,3)--(6,2.5)--(8.5,3.75)--cycle
      (8,6)--(6,7)--(6,7.5)--(8.5,6.25)--cycle
      (4,4)--(4,6)--(4.5,6.25)--(4.5,3.75)--cycle
      (6,3)--(8,4)--(8,4.5)--(5.5,3.25)--cycle
      (6,7)--(8,6)--(8,5.5)--(5.5,6.75)--cycle
      ;
    \fill
      (8,4) circle (1.2mm)
      (4,4) circle (1.2mm)
      (6,7) circle (1.2mm)
      ;  
    \draw[line width = .5mm]
              (2,1)--(4,0)--(6,1)--(8,0)--(10,1)--(10,3)--(12,4)--(12,6)--(10,7)--(10,9)--(8,10)--(6,9)--(4,10)--(2,9)--(2,7)--(0,6)--(0,4)--(2,3)--cycle
      ;
    \draw[thin]
      (2,3)--(4,4) (2,7)--(4,6)--(4,4)--(6,3)--(6,1) (4,6)--(6,7) (6,3)--(8,4) (6,9)--(6,7)--(8,6)--(8,4)--(10,3) (8,6)--(10,7)
     (7.5,4.3) node {\large $A$}
     (8.8,3.2) node {$a$}
     (4.5,4.3) node {\large $B$}
     (7.2,3.2) node {$b$}
     (4.8,3.2) node {$b$}
     (2.8,3.8) node {$b$}
     (5.5,2) node {$b$}
     (6,6.5) node {\large $C$}
     (3.6,5) node {$c$}
     (8.4,5) node {$c$}
     (3.2,6.8) node {$c$}
     (4.8,6.8) node {$c$}
     (7.2,6.8) node {$c$}
     (9.2,6.2) node {$c$}
     (6.4,8) node {$c$}
     ;
\end{scope}
    \end{tikzpicture}
  }
  ;
\end{tikzpicture}
\end{center}
\caption{A tiling of the plane by congruent clusters of seven regular hexagons and perturbations within a single cluster.\label{fig:1}}
\end{figure}
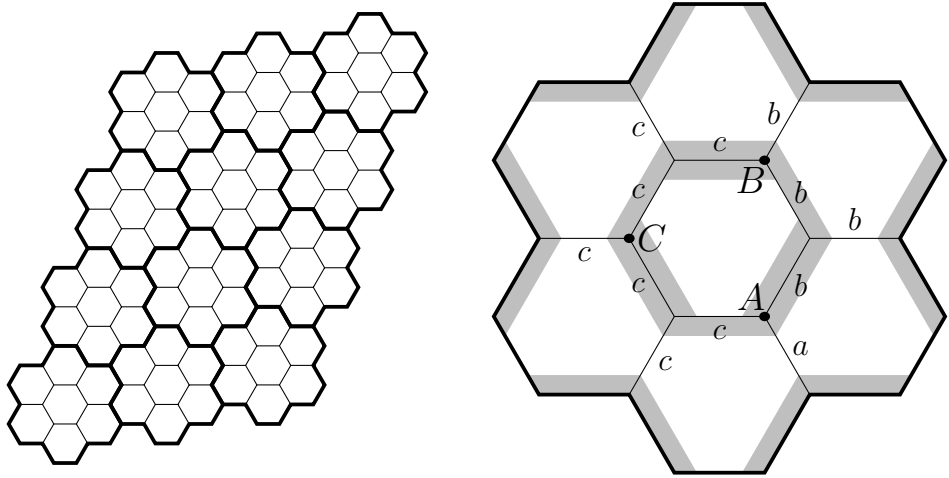
The final tiling shall be obtained by perturbing the original one within every single cluster.
That perturbation is done inductively, cluster by cluster. In every cluster we mark $21$ side figures as in the right-hand part of Figure~\ref{fig:1}. Then we perturb the coordinates of the points $A$, $B$ and $C$ by sufficiently small amounts. We keep the boundary of the cluster as well as the topology of the tiling within the cluster unchanged. The choice of $A$ fixes the side marked by $a$. By condition (ii), the choice of $B$ fixes the sides marked by $b$, and $C$ fixes those marked by $c$. This ensures (i), (ii) and (iii). 

In order to obtain (iv), we do the perturbations within each cluster such that the marked side figures of the resulting tiles are mutually incongruent as well as incongruent to those in the previously perturbed clusters. For side figures on the boundary of the cluster this incongruence is based on the sizes of their inner angles. The possibility of changing side lengths gives even more flexibility for the shapes of the other side figures. In particular, $A$ is chosen such that the two angles of the side marked by $a$ with the boundary of the cluster differs from all angles in previously fixed side figures. Then $B$ is chosen such that the two sides marked by $b$ who emanate from the boundary have angles with the boundary again different from all previously fixed angles, and that the sides with mark $b$ in the interior of the cluster constitute an angle different from all previously fixed angles and get mutually different sizes that are different from the side lengths in all previously fixed side figures. The same has to be done with $C$. It is a useful strategy to fix $A,B,C$ such that the perturbations of all critical angles and lengths are smaller than all those who were obtained in previously perturbed clusters.
\end{proof}

\begin{lem}\label{lem:2}
For every $\mu > 0$, there exists $\varepsilon > 0$ such that the following is satisfied. Let $H$ be a hexagon that is $\varepsilon$-close to a regular hexagon of side length $1$ and let three side figures over non-adjacent sides of $H$ be marked. Then $H$ splits into three pentagons of the same area and of perimeter $u=2+3\sqrt{2}-\sqrt{6}$, each one possessing one of the marked side figures of $H$. The sizes of the inner angles of each pentagon, in successive order, are $\mu$-close to $\frac{2\pi}{3}$, $\frac{2\pi}{3}$, $\frac{7\pi}{12}$, $\frac{2\pi}{3}$ and $\frac{5\pi}{12}$, where the sides between inner angles of sizes $\mu$-close to $\frac{2\pi}{3}$ represent the marked side figures of $H$.
\end{lem}

\begin{proof}
Figure~\ref{fig:2} illustrates a dissection of a hexagon $H$ with three marked side figures into three pentagons $P_1$, $P_2$ and $P_3$.
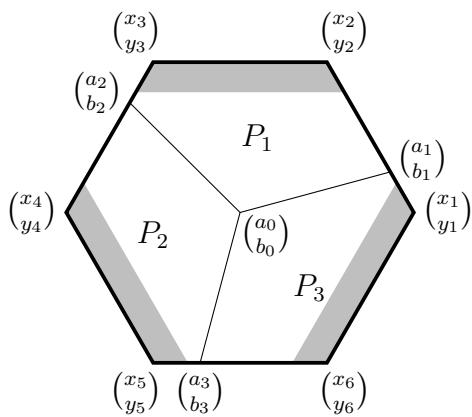
\begin{figure}
\begin{center}
\begin{tikzpicture}[xscale=1.155,yscale=1]

\fill[lightgray]
  (3.6,0)--(4,0)--(5,2)--(4.8,2.4)--cycle
  (2.4,0)--(2,0)--(1,2)--(1.2,2.4)--cycle
  (1.8,3.6)--(2,4)--(4,4)--(4.2,3.6)--cycle
  ;

\draw[line width = .5mm] 
  (2,0)--(4,0)--(5,2)--(4,4)--(2,4)--(1,2)--cycle
  ;
  
\draw  
  (2.54,0)--(3,2)--(1.73,3.46)
  (3,2)--(4.73,2.54)
  (3.3,1.7) node {$\genfrac(){0pt}{}{a_0}{b_0}$}
  (5.4,2) node {$\genfrac(){0pt}{}{x_1}{y_1}$}
  (5.1,2.7) node {$\genfrac(){0pt}{}{a_1}{b_1}$}
  (4.2,4.4) node {$\genfrac(){0pt}{}{x_2}{y_2}$}
  (1.8,4.4) node { $\genfrac(){0pt}{}{x_3}{y_3}$}
  (1.35,3.6) node {$\genfrac(){0pt}{}{a_2}{b_2}$}
  (.6,2) node {$\genfrac(){0pt}{}{x_4}{y_4}$}
  (1.8,-.4) node {$\genfrac(){0pt}{}{x_5}{y_5}$}
  (2.54,-.4) node {$\genfrac(){0pt}{}{a_3}{b_3}$}
  (4.2,-.4) node {$\genfrac(){0pt}{}{x_6}{y_6}$}
  (3.2,3) node {$P_1$}
  (2,1.7) node {$P_2$}
  (3.8,1) node {$P_3$}
  ;
\end{tikzpicture}
\end{center}
\caption{A dissection of a nearly regular hexagon into pentagons.\label{fig:2}}
\end{figure}
It is described by the vector of parameters
\[
v=(x_1,\ldots,x_6,y_1,\ldots,y_6,a_0,\ldots,a_3,b_0,\ldots,b_3) \in \mathbb{R}^{20}.
\]
The displayed situation is the unperturbed one, based on $v=v^0=\left(x_1^0,\ldots,b_3^0\right)$ with
\begin{eqnarray*}
&\left(x_i^0,y_i^0\right)=\left(\cos \frac{(i-1)\pi}{3}, \sin \frac{(i-1)\pi}{3}\right) \text{ for } i=1,\ldots,6,\quad \left(a_0^0,b_0^0\right)=(0,0),\\ 
& \left(a_i^0,b_i^0\right)=\frac{3\sqrt{2}-\sqrt{6}}{2}\left(\cos\frac{(1+8(i-1))\pi}{12},\sin\frac{(1+8(i-1))\pi}{12}\right) \text{ for } i=1,2,3,\\
& \text{i.e.}\quad \left(a_1^0,b_1^0\right)=\left(\frac{\sqrt{3}}{2},\frac{-3+2\sqrt{3}}{2}\right),\quad
\left(a_2^0,b_2^0\right)=\left(\frac{-3+\sqrt{3}}{2},\frac{3-\sqrt{3}}{2}\right),\quad
\left(a_3^0,b_3^0\right)=\left(\frac{3-2\sqrt{3}}{2},\frac{-\sqrt{3}}{2}\right).
\end{eqnarray*}
Then $H$ is regular with side length $1$, the angles between the sides $\left(\genfrac{}{}{0pt}{}{a_0^0}{b_0^0}\right)\left(\genfrac{}{}{0pt}{}{a_i^0}{b_i^0}\right)$, $i=1,2,3$, and the respective diagonals of $H$ are of size $\frac{\pi}{12}$, and the pentagons are congruent and have perimeter $u$. We shall see that, for all $(x_1,\ldots,x_6,y_1,\ldots,y_6)$ sufficiently close to $\left(x_1^0,\ldots,x_6^0,y_1^0,\ldots,y_6^0\right)$, there exist values $(a_0,\ldots,a_3,b_0,\ldots,b_3)$ close to $\left(a_0^0,\ldots,a_3^0,b_0^0,\ldots,b_3^0\right)$ describing a dissection into pentagons $P_1$, $P_2$, $P_3$ of the same area and with perimeter $u$.

We formulate our claim in terms of equations $f_i(v)=0$, $i=1,\ldots,8$. First note that $\genfrac(){0pt}{}{a_1}{b_1}$ lies on the straight line through $\genfrac(){0pt}{}{x_1}{y_1}$ and $\genfrac(){0pt}{}{x_2}{y_2}$; i.e., $\det\left(\genfrac(){0pt}{}{a_1}{b_1}-\genfrac(){0pt}{}{x_1}{y_1},\genfrac(){0pt}{}{x_2}{y_2}-\genfrac(){0pt}{}{x_1}{y_1}\right)=0$. This amounts to
\begin{equation}\label{eq:f1}
f_1(v)=(a_1-x_1)(y_2-y_1)-(x_2-x_1)(b_1-y_1)=0.
\end{equation}
Similarly,
\begin{align}
\label{eq:f2}
f_2(v)=(a_2-x_3)(y_4-y_3)-(x_4-x_3)(b_2-y_3)=0, \\
\label{eq:f3}
f_3(v)=(a_3-x_5)(y_6-y_5)-(x_6-x_5)(b_3-y_5)=0.
\end{align}
The (signed) area of a convex pentagon $p_1,\ldots,p_5$ is 
\[
\text{area}(p_1,\ldots,p_5)= \frac{1}{2}(\det(p_2-p_1,p_3-p_1)+\det(p_3-p_1,p_4-p_1)+\det(p_4-p_1,p_5-p_1)).
\]
The equality of the areas of the pentagons $P_1$, $P_2$ and $P_3$ of our dissection is expressed by the equations
\begin{align}
\label{eq:f4}
f_4(v)=\text{area}\left(\genfrac(){0pt}{1}{a_0}{b_0},\genfrac(){0pt}{1}{a_2}{b_2},\genfrac(){0pt}{1}{x_4}{y_4},\genfrac(){0pt}{1}{x_5}{y_5},\genfrac(){0pt}{1}{a_3}{b_3}\right)-\text{area}\left(\genfrac(){0pt}{1}{a_0}{b_0},\genfrac(){0pt}{1}{a_1}{b_1},\genfrac(){0pt}{1}{x_2}{y_2},\genfrac(){0pt}{1}{x_3}{y_3},\genfrac(){0pt}{1}{a_2}{b_2}\right)=0, \\
\label{eq:f5}
f_5(v)=\text{area}\left(\genfrac(){0pt}{1}{a_0}{b_0},\genfrac(){0pt}{1}{a_3}{b_3},\genfrac(){0pt}{1}{x_6}{y_6},\genfrac(){0pt}{1}{x_1}{y_1},\genfrac(){0pt}{1}{a_1}{b_1}\right)-\text{area}\left(\genfrac(){0pt}{1}{a_0}{b_0},\genfrac(){0pt}{1}{a_1}{b_1},\genfrac(){0pt}{1}{x_2}{y_2},\genfrac(){0pt}{1}{x_3}{y_3},\genfrac(){0pt}{1}{a_2}{b_2}\right)=0.
\end{align}
The perimeter of a pentagon $p_1,\ldots,p_5$ is 
\[
\text{perim}(p_1,\ldots,p_5)=\|p_2-p_1\|+\|p_3-p_2\|+\|p_4-p_3\|+\|p_5-p_4\|+\|p_1-p_5\|,
\]
where $\|\cdot\|$ denotes the Euclidean norm. Now the coincidences of the perimeters of $P_1$, $P_2$ and $P_3$ with $u$ read as
\begin{align}
\label{eq:f6}
f_6(v)=\text{perim}\left(\genfrac(){0pt}{1}{a_0}{b_0},\genfrac(){0pt}{1}{a_1}{b_1},\genfrac(){0pt}{1}{x_2}{y_2},\genfrac(){0pt}{1}{x_3}{y_3},\genfrac(){0pt}{1}{a_2}{b_2}\right)-u=0,\\
\label{eq:f7}
f_7(v)=\text{perim}\left(\genfrac(){0pt}{1}{a_0}{b_0},\genfrac(){0pt}{1}{a_2}{b_2},\genfrac(){0pt}{1}{x_4}{y_4},\genfrac(){0pt}{1}{x_5}{y_5},\genfrac(){0pt}{1}{a_3}{b_3}\right)-u=0, \\
\label{eq:f8}
f_8(v)=\text{perim}\left(\genfrac(){0pt}{1}{a_0}{b_0},\genfrac(){0pt}{1}{a_3}{b_3},\genfrac(){0pt}{1}{x_6}{y_6},\genfrac(){0pt}{1}{x_1}{y_1},\genfrac(){0pt}{1}{a_1}{b_1}\right)-u=0.
\end{align}

One easily checks that the equations \eqref{eq:f1}--\eqref{eq:f8} are satisfied for $v=v^0$.
Moreover, symbolic calculations of a computer algebra system, such as Maple 2019, show that
\[
\det\left(
\begin{array}{cccccc}
\frac{\partial f_1}{\partial a_0}(v^0) &
\cdots&
\frac{\partial f_1}{\partial a_3}(v^0) &
\frac{\partial f_1}{\partial b_0}(v^0) &
\cdots&
\frac{\partial f_1}{\partial b_3}(v^0) \\
\vdots && \vdots & \vdots && \vdots\\
\frac{\partial f_8}{\partial a_0}(v^0) &
\cdots&
\frac{\partial f_8}{\partial a_3}(v^0) &
\frac{\partial f_8}{\partial b_0}(v^0) &
\cdots&
\frac{\partial f_8}{\partial b_3}(v^0)
\end{array}
\right)= \frac{-162 + 486\sqrt{2} + 81\sqrt{3} - 270\sqrt{6}}{8} \ne 0.
\]
Thus the implicit function theorem says that the system of Equations \eqref{eq:f1}--\eqref{eq:f8} has a unique solution for $a_0,\ldots,a_3,b_0,\ldots,b_3$ depending on $x_1,\ldots,x_6,y_1,\ldots,y_6$ and that $\left(a_0,\ldots,a_3,b_0,\ldots,b_3\right)$ is arbitrarily close to $\left(a_0^0,\ldots,a_3^0,b_0^0,\ldots,b_3^0\right)$, whenever $(x_1,\ldots,x_6,y_1,\ldots,y_6)$ is sufficiently close to $\left(x_1^0,\ldots,x_6^0,y_1^0,\ldots,y_6^0\right)$. That is, an arbitrary, but sufficiently small perturbation of the regular hexagon with vertices $\left(\genfrac{}{}{0pt}{}{x_1^0}{y_1^0}\right),\ldots,\left(\genfrac{}{}{0pt}{}{x_6^0}{y_6^0}\right)$ provides a respectively small perturbed version of the original tiling with parameters $v^0$. Indeed, the closeness to $v^0$ and Equations \eqref{eq:f1}--\eqref{eq:f3} guarantee that the perturbed (given and dependent) parameters describe a tiling into three pentagons. Moreover, the inner angles of the pentagons are $\mu$-close to $\frac{2\pi}{3}$, $\frac{2\pi}{3}$, $\frac{7\pi}{12}$, $\frac{2\pi}{3}$ and $\frac{5\pi}{12}$, provided the perturbation is sufficiently small. Equations \eqref{eq:f4} and \eqref{eq:f5} ensure that the tiles are of the same area. By Equations \eqref{eq:f6}--\eqref{eq:f8}, the pentagons have perimeter $u$. 
\end{proof}

\begin{proof}[Proof of Theorem~\ref{thm:pentagons}]
We apply Lemma~\ref{lem:2} to all hexagons of a tiling obtained by Lemma~\ref{lem:1}. This gives a tiling of the plane by convex pentagons of area $\frac{\sqrt{3}}{2}$ and of perimeter $2+3\sqrt{2}-\sqrt{6}$. If $\mu$ is sufficiently small, congruence of two pentagons implies that their side figures at the sides connecting inner angles of sizes $\mu$-close to $\frac{2\pi}{3}$ must be congruent. But all these side figures are marked side figures of the tiling given by Lemma~\ref{lem:1} and in turn mutually incongruent.
\end{proof}


\section*{Acknowledgments}
Both authors express their gratitude to R.\ Nandakumar for providing interesting problems.



\end{document}